\documentclass[12pt]{amsart}

\numberwithin{equation}{section}

\usepackage{amsmath,amsthm,amsfonts,eucal,}
\usepackage{xypic}
\usepackage{graphicx}
\usepackage{amssymb}

\def\cb{{\mathcal B}}

\def\cf{{\mathcal F}}

\def\ch{{\mathcal H}}

\def\co{{\mathcal O}}

\def\cs{{\mathcal S}}

\def\ga{{\mathfrak A}} 
\def\gb{{\mathfrak B}}

\def\gg{{\mathfrak G}}

\def\gam{{\mathfrak M}}

\def\gar{{\mathfrak R}}

\def\bc{{\mathbb C}}

\def\bl{{\mathbb L}}

\def\bm{{\mathbb M}}
\def\bn{{\mathbb N}}

\def\bp{{\mathbb P}}
\def\br{{\mathbb R}}

\def\bz{{\mathbb Z}}

\def\a{\alpha}

\def\g{\gamma}  \def\G{\Gamma}
\def\d{\delta}

\def\l{\lambda} \def\L{\Lambda}

\def\p{\pi}

\def\s{\sigma} 
\def\t{\tau}
\def\f{\varphi}  \def\F{\Phi}
  
\def\om{\omega} 
\def\z{\zeta}

\newtheorem{thm}{Theorem}[section]

\newtheorem{lem}[thm]{Lemma}

\newtheorem{prop}[thm]{Proposition}

\theoremstyle{definition}
\newtheorem{rem}[thm]{Remark}
\newtheorem{defin}[thm]{Definition}

\def\min{\mathop{\rm min}}

\def\idd{{1}\!\!{\rm I}}

\begin{document}
\title[Tail algebras]
{Tail algebras for monotone and $q$-deformed exchangeable stochastic processes}
\author{Vitonofrio Crismale}
\address{Vitonofrio Crismale\\
Dipartimento di Matematica\\
Universit\`{a} degli studi di Bari\\
Via E. Orabona, 4, 70125 Bari, Italy}
\email{\texttt{vitonofrio.crismale@uniba.it}}

\author{Stefano Rossi}
\address{Stefano Rossi\\
Dipartimento di Matematica\\
Universit\`{a} degli studi di Bari\\
Via E. Orabona, 4, 70125 Bari, Italy}
\email{\texttt{stefano.rossi@uniba.it}}

\begin{abstract}

We compute the tail algebras of exchangeable monotone stochastic processes. This allows us to
prove the analogue of de Finetti's theorem for this type of processes. In addition,
since the vacuum state on the 
$q$-deformed $C^*$-algebra is the only exchangeable state when $|q|<1$, we draw our attention to its
tail algebra, which turns out to obey a zero-one law. 

\end{abstract}

\maketitle
\section{Introduction}

In Classical Probability de Finetti's theorem characterizes exchangeability for a sequence of
random variables in terms of conditional independence and identical distribution
with respect to the tail algebra of the sequence itself, see {\it e.g.} \cite{Ka}.
Furthermore, exchangeability is equivalent to spreadability by virtue of a well-known
result due to Ryll-Nardzewski, \cite{R}.
Now it is relatively easy to guess what
the generalized statements of this couple of results should be within
the various non-commutative formalisms offered by Quantum Probability.
Yet the methods employed to prove them can often be far from
straightforward. Interestingly, surprises can occur when
the classical results cease to hold in suitably chosen models.
For instance, on the CAR algebra there exist spreadable states
that are not exchangeable, \cite{CRbis}.
Furthermore, in the $W^*$-formalism adopted by K\"{o}stler in
\cite{K}, examples have been  found of spreadable quantum processes which
nevertheless fail to be exchangeable. Sticking to this formalism, though, 
the sought generalizations can still be obtained  as long as exchangeability and spreadability are
replaced by their quantum analogues, see \cite{KS} and \cite{Cu} respectively, where
the symmetries are implemented by natural coactions of quantum groups.
It is also possible to arrive at quantum generalizations of de Finetti's theorem
in which the involved symmetries continue to be implemented by classical groups. However, this will require restricting the class of
quantum processes or variables addressed.
For example, by following this approach a de Finetti-type theorem can be proved
for Boolean \cite{Fbo}
and Fermi processes  \cite{CFCMP, CrFid}:  all processes of both types are exchangeable if and only if they are independent and identically distributed with respect to the tail algebra.
Similar results can actually be obtained in the framework
of quasi-local algebras  \cite{CRZ}, although  the analogy with the classical setting is looser, inasmuch as
the states of these $C^*$-algebras do not necessarily correspond to a stochastic process;
for the correspondence bewteen states and stochastic processes see \cite{CrFid, CrFid2}.
Even so, for infinite $\bz_2$-graded tensor products, which are of course examples of quasi-local
algebras, the analogy is complete, see \cite{CRZ, Fid}.

Pursuing this line of research, the present paper mainly but not only deals with monotone processes
in the $C^*$-algebraic formalism.
In Section \ref{prel} we first establish some notation from $C^*$-algebra theory and collect the necessary definitions from  the theory of quantum stochastic processes.
The bulk of the work is contained in Section \ref{monotone}, at the begin of which
we provide a rather quick yet self-contained account of the monotone algebra $\gam$
as the concrete unital $C^*$-algebra generated by annihilation and creation operators on the monotone
Fock space. After recalling how states of $\gam$ correspond to monotone processes, we move on to
complete the study of symmetric states initiated in \cite{CFG}.
As shown there, although the group of finite permutations on a countable set 
does not act on $\gam$ through automorphisms, exchangeability
of a monotone  process can still be seen as an invariance property
of the corresponding state, albeit with respect  to maps that are only unital, linear
and $*$-preserving.
In the aforementioned paper the set of these states is seen to coincide with  the set of spreading-invariant states, being
both equal to the segment whose endpoints are the vacuum state and the state at infinity.\\
The tail algebras of monotone symmetric states are described in Proposition \ref{tailalgebra}.
These are never trivial, apart from the tail algebra of the state at infinity. 
More precisely, the tail
algebra of the vacuum state is isomorphic with $\bc\oplus\bc$, where the two measure atoms
account for the projection onto the vacuum vector and its orthogonal projection. 
However, monotone tail algebras of symmetric states are always commutative and finite-dimensional. Moreover, they are strictly
contained in the stationary algebra, as observed in Remark \ref{stationary}.
Unlike the Boolean case analyzed in \cite{Fbo}, monotone
exchangeable algebras may even not be defined since it is not known whether the linear maps
alluded to above, which are only defined on a dense $*$-algebra of $\gam$, extend to the whole
$C^*$-algebra. Nevertheless, in Proposition \ref{rotalgebra} we show that the self-adjoint
subspace of the elements in the dense $*$-algebra fixed  by all these maps reduces to the scalars.
A version of de Finetti's theorem for monotone processes does hold. Indeed, in 
Theorem \ref{definetti} we prove that a monotone process is
exchangeable if and only if it is conditionally independent and identically
distributed with respect to its tail algebra.\\
In Section \ref{qdeformed} we  focus on processes on the so-called $q$-deformed algebra, with
$|q|< 1$. Unlike the monotone case, the corresponding dynamical
systems are now uniquely ergodic,  \cite{DyF}. However, what we believe is worth pointing out is that
enjoying good ergodic properties still gives no guarantees at the level of the stochastic process itself, insofar as
the Hewitt-Savage \cite{HS} and the Olshen \cite{Ols} theorem 
fail for $q$-deformed processes. More precisely, we show that the tail algebra of the vacuum state is trivial, Proposition \ref{trivialq}. Moreover, the tail, exchangeable, and stationary algebras of the vacuum state are
seen to be different from one another in Remark \ref{remq}. \\
Finally, for the reader's convenience we defer to the appendix a possibly known result
which shows how to recover classical (countable) stochastic processes in the 
$C^*$-algebra framework.

\section{Quantum stochastic processes and invariant states }\label{prel}

We collect here some basic definitions from  $C^*$-algebra theory as well as establishing
some notation which will be used consistently throughout the paper.\\ 
Given a collection of unital $C^*$-algebras $\{\ga_j\}_{j\in J}$ indexed by a set $J$, their unital free product $C^*$-algebra $\ast_{j\in J} \ga_j$ ({\it cf.} \cite{A}) is the unique (up to isomorphism) unital $C^*$-algebra determined by the following universal property: there are unital monomorphisms $i_j:\ga_j\rightarrow \ast_{j\in J} \ga_j$ such that for any unital $C^*$-algebra $\gb$ and unital morphisms $\F_j:\ga_j\rightarrow \gb$, there exists a unique unital homomorphism $\F:\ast_{j\in J} \ga_j\rightarrow \gb$ such that $\Phi\circ i_j=\Phi_j$ for every $j\in J$.
When $\ga_j=\ga$ for every $j\in J$, we denote the corresponding free product simply as $\ast_{J} \ga$. Details are to be found  in  \cite{CrFid, CrFid2, VDN}.\\

We now recall the definition of a (realization of a)  {\it quantum stochastic process} labelled by the index set $J$.
This is a quadruple
$\big(\ga,\ch,\{\iota_j\}_{j\in J},\xi\big)$, where $\ga$ is a (unital) $C^{*}$-algebra, referred to as the sample
algebra of the process, $\ch$ is a Hilbert space,
whose inner product is denoted by $\langle\cdot,\cdot \rangle$,
the maps $\iota_j$ are (unital) $*$-homomorphisms from $\ga$ to $\cb(\ch)$, and
$\xi\in\ch$ is a unit vector, which is cyclic for  the von Neumann algebra
$\bigvee_{j\in J}\iota_j(\ga)$.\\
A quadruple as above can equivalently be assigned through a state $\varphi$ on
the free product $C^*$-algebra $\ast_{J} \ga$. Indeed, if one starts 
with a stochastic process, then a state
$\varphi$ on the free product $\ast_{J} \ga$ can be defined by setting
$\varphi(i_j(a)):=\langle\iota_j(a)\xi, \xi\rangle$, $a\in\ga$ and $j\in J$. Rather interestingly, all states on the free product $\ast_{J} \ga$ arise in this way,
see  \cite[Theorem 3.4]{CrFid}.
Phrased differently, starting now with a state $\f\in\cs\big(\ast_{J}\ga\big)$, it is possible
to recover a stochastic process by looking at the GNS representation $(\p_\f, \ch_\f, \xi_\f)$ of
$\f$. Indeed, for every $j\in J$ we can set $\iota_j(a):=\pi_\f(i_j(a))$, $a\in\ga$, so as to get quadruple
$\big(\ga,\ch_\f,\{\iota_j\}_{j\in J},\xi_\f\big)$. 

The above definition of a quantum stochastic process is general enough to include
all classical cases as realised in the Kolmogorov consistency theorem. Indeed, suppose one is given a  family of finite-dimensional 
 distributions $\mu_{j_1, j_2, \ldots, j_k}\in P(\br^k)$, for any $k\in\bn$ and indices
$j_1, j_2, ..., j_k\in J$.
Consider the Tychonoff product $\br^J$, whose points are denoted by
$x=(x_j)_{j\in J}$. For any $k\in\bn$, $j_1, j_2, \ldots, j_k\in J$, and $A_1, A_2, \ldots, A_k$ Borel subsets of $\br$, we define
a corresponding cylinder of $\br^J$ as
$$C_{j_1, j_2, \ldots, j_k}^{A_1, A_2, \ldots, A_k}:=\big\{x\in \br^J \mid x_{j_1}\in A_1, x_{j_2}\in A_2, \ldots, x_{j_k}\in A_k\big\}.$$
As is known, if the finite-dimensional distributions satisfy the so-called consistency conditions of
Kolmogorov's theorem (see {\it e.g.} \cite{Bi}), then a measure $\mu$ can be defined on the $\sigma$-algebra
$\mathfrak{C}$
generated by cylinders by 
\begin{equation*}\label{cylmeasure}
\mu\big(C_{j_1, j_2, \ldots, j_k}^{A_1, A_2, \ldots, A_k}\big)= \mu_{j_1, j_2, \ldots, j_k}
(A_1\times A_2\times\ldots\times A_k).
\end{equation*}
Denote by $X_j$ the $j$-th coordinate function on $\br^J$, namely
$X_j(x)=x_j$, $x\in \br^J$.
The functions $\{X_j\}_{j\in J}$ provide a stochastic process on the probability space
$(\br^J, \mathfrak{C}, \mu)$ having $\{\mu_{j_1, j_2, \ldots, j_k}\}$ as finite-dimensional distributions.\\
The above construction can also be reinterpreted in terms of representations of a given sample algebra.
To this end, take $\ga=C_0(\br)$.
With $\ch:=L^2(\br^J, \mu)$, we consider a  family of $*$-homomorphisms of $\ga$ 
to $\mathcal{B}(\ch)$ defined as
$$
\iota_j(f):= f(X_j), \,  \quad f\in C_0(\br)\,,
$$
where $f(X_j)$ is the bounded operator acting on $L^2(\br^J, \mu)$  by multiplication
by the function $f(X_j)$.\\
Denoting by $1\in L^2(\br^J, \mu)$ the function  equal to 
$1$ $\mu$-a.e., we obviously have that $(C_0(\br), L^2(\br^J, \mu), \{\iota_j\}_{j\in J}, 1)$ is a stochastic process and
corresponds to the realization of the process provided by  Kolmogorov's theorem
through the coordinate functions $X_j$.\\
Conversely, a process is classical exactly when both the sample algebra
$\ga$ and $\bigvee_{j\in J}\iota_j(\ga)$ are commutative, see {\it e.g.} Appendix.\\ 

By a $C^*$-dynamical system we mean a triplet $(\ga, M,\Gamma)$, where $\ga$ is a unital $C^*$-algebra, $M$ is a monoid (unital semigroup), and $\G$ is a
representation $M \ni g\mapsto \G_g\in{\rm End}(\ga)$  by unital $*$-endomorphisms of $\ga$ such that
$\G_{gh}=\G_g \circ\G_h$, for any $g,h \in M$.
When $M=G$ is a group, in the corresponding $C^*$-dynamical system $(\ga, G, \a)$,  each $\a_g$ is a $*$-automorphisms of
$\ga$. In this case, one typically speaks of reversible dynamics, whereas dissipative dynamics corresponds to having non-invertible maps, see {\it e.g.} \cite{BR1}.
By $\cs(\ga)$ we denote the convex set of all states (normalized, positive, linear functionals) on $\ga$. Since the latter is assumed unital, $\cs(\ga)$  is weakly $*$-compact.
For any given $(\ga, M,\Gamma)$, we can define the convex subset $\cs_M(\ga)\subset\cs(\ga)$ of those states of
$\ga$ which are invariant under the action $\Gamma$ of $M$ as

$$
\cs_M(\ga):=\big\{\f\in\cs(\ga)\mid\f\circ\G_g=\f\,,\,\,g\in M\big\}\,.
$$
This is easily seen to be compact in the weak-$*$ topology.\\
When $M=G$ is a group, for any $\varphi\in
\cs_G(\ga)$ the corresponding fixed von Neumann algebra
$\pi_\varphi(\ga)''_G$ is defined as $\{T\in\cb(\ch_\varphi): U_g T U_g^*=T\,\, \textrm{for all}\, g\in G\}$, where
for $g\in G$ the operator $U_g$ is the unitary implementator  of $\a_g$, that is 
$U_g\pi_\f(a)U_g^*=\pi_\f(\a_g(a))$, $a\in \ga$.\\

Comparing the distributional symmetries involved in the present paper requires considering some algebraic structures on $\bz$.\\ 
We first consider the group generated by the one-step shift $\t(i):=i+1$ of the integers $\bz$, which is canonically identified with $\bz$ itself.\\
Let us now denote by $\bp_n$
the symmetric
group of order $n\in\bn$.
For any $n\geq 2$ we identify $\bp_{n-1}$ with the subgroup in $\bp_n$ fixing the $n$th element.
Thus, one can take the direct limit of the groups above, which may be easily identified with the group of all permutations on the integers leaving fixed all  but finitely many elements, denoted by $\bp_\bz$.\\
In addition, we consider the monoid $\bl_\bz$ of all strictly increasing maps $g:\bz\to\bz$.\\ 
By universality, the action on $\bz$ of the above algebraic structures can be lifted to the free $C^*$-algebra of the infinite free product $\ast_{\bz} \ga$ of a given sample algebra $\ga$.
We denote by $(\ast_{\bz} \ga,\t,\bz)$, $(\ast_{\bz} \ga,\a,\bp_\bz)$ and
$(\ast_{\bz} \ga,\lambda,\bl_\bz)$
the corresponding dynamical systems.  Note that $(\ast_{\bz} \ga,\lambda,\bl_\bz)$
is a dissipative dynamical system. However,  for every $h\in\bl_\bz$ the endomorphism $\lambda_h$ is injective since 
$\bl_\bz$ has left inverses. 
One clearly has $\cs_{\bp_\bz}(\ast_{\bz}\ga)\subseteq \cs_{\bl_\bz}(\ast_{\bz}\ga)\subseteq\cs_{\bz}(\ast_{\bz}\ga)$.\\
%

A stochastic process  $\big(\ga,\ch,\{\iota_j\}_{j\in \bz},\xi\big)$ is called
\begin{itemize}
\item[{\bf-}]  {\it spreadable} if the corresponding state  is invariant under the action on $\ast_{\bz} \ga$
of the monoid $\bl_\bz$, and the state itself is  said to be spreadable;
\item[{\bf-}] {\it exchangeable} if the corresponding state  is invariant under the action on $\ast_{\bz} \ga$
of the group $\bp_\bz$, and the state itself is  said to be symmetric;
\item[{\bf-}] {\it stationary} (or shift-invariant) if the corresponding state  is invariant under the action on $\ast_{\bz} \ga$
of the group $\bz$, and the state itself is  said to be stationary.
\end{itemize}

In some classes of processes, such as Boolean \cite{Fbo} and monotone, 
the representations $\iota_j$ are subject to additional constraints of various types.
These can be embodied by passing from the
free product $\ast_{\bz} \ga$ to its corresponding quotient
$C^*$-algebra. In this way, a process may directly
be described in terms of a state of the quotient algebra.\\

Going back to the general framework, we would like to point out that
automorphisms of the free product algebra provide quite a natural way to obtain
a new  process out of a given stochastic process. Indeed,
let $(\ga, \{\iota_j\}_{j\in J}, \ch, \xi)$ be a given stochastic process and let $\varphi$
be the corresponding state of $\ast_J \ga$.
If $\Phi$ is a $*$-automorphism $\ast_J \ga$, we can consider a new process
$(\ga, \{\iota_j^\Phi\}_{j\in J}, \ch, \xi)$ by defining $\iota_j^\Phi$ 
on the same sample algebra as
\begin{equation}\label{funproc}
\iota_j^\Phi:=\pi_\varphi\circ\Phi\circ i_j, \,j \in J
\end{equation}
In particular, one can choose an automorphism
$\Phi$ that factors as $\ast_J \Psi$, where $\Psi$ is a
fixed $*$-automorphism of the sample algebra $\ga$.
In this case, for any $j\in\bz$ one has 
$i_j\circ\Psi=\F\circ i_j$, and the transformed process in \eqref{funproc} is simply given by
\begin{equation}\label{funprocbis}
\iota_j^\Phi=\iota_j\circ\Psi, \, j\in J.
\end{equation}

Classically, this corresponds to passing from a process
$\{X_j\}_{j\in J}$ to $\{F(X_j)\}_{j\in J}$, where
$F$ is a homeomorphism of the spectrum of the sample algebra, {\it i.e.}
$\Psi(f)=f\circ F$, $f\in C(\sigma(\ga))$.\\

\begin{rem}
A stochastic process $\big(\ga,\ch,\{\iota_j\}_{j\in \bz},\xi\big)$ is stationary (exchangeable, spreadable)
if and only if $\big(\ga,\ch,\{\iota_j^\Phi\}_{j\in \bz},\xi\big)$ is so, where $\Phi$ is an automorphism
of $\ast_\bz\ga$ of the form $\ast_\bz \Psi$, for a fixed automorphims $\Psi$ of $\ga$,
as easily seen from \eqref{funprocbis}.
\end{rem}
If $\Phi$ does not factor as $\ast_\bz\Psi$, one may well start with
an exchangeable process and end up with a 
process  which is no longer exchangeable.
This is already seen with two classical variables $X_1, X_2$.
For instance, if $X_1, X_2$ are Bernoulli distributed with
parameter $0<p<1$ different from $\frac{1}{2}$, then $1-X_1$ and $X_2$  no longer have the same distribution, for
$1-X_1$ is a Bernoulli variable with parameter $1-p$. 
Note that the transformation of the two variables 
$X_1, X_2$ is  induced by the homeomorphism
$F(x_1, x_2)=(1-x_1, x_2)$.\\

What is more, when the sample $C^*$-algebra $\ga$ is singly generated, the distributional symmetries 
dealt with in the present work
can be read directly in terms of the variables $\{\iota_j(a_0)\}_{j\in\bz}$, where $a_0$ is any
generator of $\ga$.
For instance,
the process $\{\iota_j \}_{j\in\bz}$ will be stationary (spreadable, exchangeable) if and only if the corresponding
variables $\{\iota_j(a_0^\natural)\}_{ j\in\bz}$ are so, where $\natural$ is either $1$ or $*$.
To take but one example, let us see how exchangeability can be dealt with. 
Clearly, we only need to prove that exchangeability
of the variables carries over to the whole process regardless of the choice of the generator
$a_0$. Exchangeability
of the variables amounts to
$$
\langle \iota_{j_1}(a_0^\natural)\cdots\iota_{j_k}(a_0^\natural) \xi, \xi\rangle=\langle \iota_{\s(j_1)}(a_0^\natural)\cdots\iota_{\s(j_k)}(a_0^\natural) \xi, \xi\rangle
$$
for all $k\in\bn$, $j_1, j_2, \ldots, j_k\in\bz$, and for all finite permutations $\s$.
 Since repetitions of the indices are allowed, the above equality also holds as 
$$
\langle \iota_{j_1}(m_1)\cdots\iota_{j_k}(m_k) \xi, \xi\rangle=\langle \iota_{\s(j_1)}(m_1)\cdots\iota_{\s(j_k)}(m_k) \xi, \xi\rangle
$$
where $m_h$, $h=1, 2, \ldots, k$, are (possibly noncommutative) monomials in $a_0$ and $a_0^*$. 
Then by linearity and density the conclusion
is reached.\\

\begin{rem}
As a direct application of the considerations above, we see that the variables $\{\iota_j(a_0^\natural)\}_{ j\in\bz}$ are
exchangeable (spreadable, stationary) if and only $\{\iota_j(\Psi(a_0^\natural))\}_{ j\in\bz}$ are exchangeable (spreadable, stationary), where
$\Psi$ is any $*$-automorphism of $\ga$.
\end{rem}

\section{The case of monotone processes}\label{monotone}

For $k\geq 1$, set $I_k:=\{(i_1,i_2,\ldots,i_k) \mid i_1< i_2 < \cdots <i_k, i_j\in \mathbb{Z}\}$. The discrete monotone Fock space is the Hilbert space $\cf_m:=\bigoplus_{k=0}^{\infty} \ch_k$, where for every $k\geq 1$, $\ch_k:=\ell^2(I_k)$, and $\ch_0=\mathbb{C}\zeta$, $\zeta$ being the Fock vacuum. Borrowing the terminology from physicists' parlance, we call each
$\ch_k$ the $k$-particle space and denote by $\cf^0_m$ the total set of finite particle vectors in $\cf_m$, that is
$$
\cf^0_m:=\bigg\{\sum_{n=0}^{\infty} c_n\xi_n: \xi_n\in \ch_n,\,\, c_n\in \mathbb{C}\,\,\textrm{s.t.}\,\, c_n=0\,\,\, \text{for all $n$ but a finite set}  \bigg\} \,.
$$
The canonical basis of the discrete monotone Fock space is obtained in the following way. If $(i_1,i_2,\ldots,i_k)\in I_k$ is an increasing sequence  of integers, we denote by $e_{(i_1,i_2,\ldots,i_k)}\in \ell^2(I_k)$ the square summable sequence that is always zero but at $(i_1, i_2, \ldots, i_k)$, where
it is $1$. The vector corresponding to the empty set
is just the vacuum $\zeta$. More often than not, however, we will simply write $e_i$ instead of $e_{(i)}$  to ease our notation.
For every $i\in \mathbb{Z}$, the monotone creation and annihilation operators are respectively given  by $a^\dag_i\zeta=e_i$, $a_i\zeta=0$ and
\begin{equation*}
a^\dagger_i e_{(i_1,i_2,\ldots,i_k)}:=\left\{
\begin{array}{ll}
e_{(i,i_1,i_2,\ldots,i_k)} & \text{if}\,\, i< i_1 \\
0 & \text{otherwise}, \\
\end{array}
\right.
\end{equation*}
\begin{equation*}
a_ie_{(i_1,i_2,\ldots,i_k)}:=\left\{
\begin{array}{ll}
e_{(i_2,\ldots,i_k)} & \text{if}\,\, k\geq 1\,\,\,\,\,\, \text{and}\,\,\,\,\,\, i=i_1\\
0 & \text{otherwise}. \\
\end{array}
\right.
\end{equation*}
Both $a^\dagger_i$ and $a_i$  can be shown to have unital norm, 
to be mutually adjoint, and to satisfy the following relations
\begin{equation}
\label{comrul}
\begin{array}{ll}
  a^\dagger_ia^\dagger_j=a_ja_i=0 & \text{if}\,\, i\geq j\,, \\
\end{array}
\end{equation}
In addition, for every $i$ the following identity holds
\begin{equation}
\label{comrul2}
a_ia^\dag_j=\delta_{i,j}\left(I-\sum_{k\leq i} a^\dag_ka_k\right)
\end{equation}
where the convergence is in the strong operator topology, and $I$ is the identity
of $\cb(\cf_m)$.
We denote by $\gam$  the concrete unital
$C^*$-subalgebra of $\cb(\cf_m)$ with unit $I$ generated by the set of all annihilators $\{a_i\mid i\in\mathbb{Z}\}$, whereas  $\gam_0$ will denote the unital $*$-algebra of $\gam$ generated by the same set. Of course, $\gam_0$ is norm dense
in $\gam$. For completeness' sake, we recall that the algebra generated by so-called position
operators $a_i+a^\dag_i$  coincides with the whole monotone algebra, see \cite[Proposition 5.13]{CFL}.\\

For the reader's convenience, we also recall the following definition from \cite{CFL}.

\begin{defin}
A word $X$ in $\gam_0$ is said to have a $\lambda$-\textbf{form} if there are $m,n\in\left\{  0,1,2,\ldots
\right\}$ and $i_1<i_2<\cdots < i_m, j_1>j_2> \cdots > j_n$ such
that
$$
X=a_{i_1}^{\dagger}\cdots a_{i_m}^{\dagger} a_{j_1}\cdots a_{j_n}\,,
$$
with $X=I$, that is the empty word if $m=n=0$. Its length is $l(X)=m+n$.

A word $X$ is said to have a $\pi$-\textbf{form} if there are $m,n\in\left\{0,1,2,\ldots
\right\}$, $k\in\mathbb{Z}$, $i_1<i_2<\cdots < i_m, j_1>j_2> \cdots > j_n$ such that $i_m<k>j_1$ and
$$
X=a_{i_1}^{\dagger}\cdots a_{i_m}^{\dagger} a_{k}a_{k}^{\dagger} a_{j_1}\cdots a_{j_n}\,.
$$
The length is now $l(X)=m+n+2$.
\end{defin}
Let $\L$ be the index set such that $\{X_\l\}_{\l\in\L}$ represents all $\l$-forms, that is
\begin{equation*}
\label{xlambda}
X_\l=a_{i_1^{(\l)}}^{\dagger}\cdots a_{i_{m(\l)}^{(\l)}}^{\dagger} a_{j_1^{(\l)}}\cdots a_{j_{n(\l)}^{(\l)}}\,,
\end{equation*}
for $i_1^{(\l)}<i_2^{(\l)}<\cdots < i_{m(\l)}^{(\l)}, j_1^{(\l)}>j_2^{(\l)}> \cdots > j_{n(\l)}^{(\l)}$, $m(\l),n(\l)\geq 0$. Since all $\l$-forms of the type $a^\dag_ia_i$ are in one-to-one correspondence with $\bz$, then $\bz\subset\L$  in a natural way. After this identification, we put $\G:=\L\setminus\bz$. In \cite[Theorem 3.4]{CFG}  the families $\{X_\l\}_{\l\in\G}$ and $\{a_ia^\dag_i\}_{i\in \mathbb{Z}}$ were proved to form a Hamel basis of $\gam_0$.\\
Following \cite{CFG}, we present the set $\G\cup \bz$ in a different way, denoting every vector of the Hamel basis of $\gam_0$ by $X_{(\l_1,\l_2)}$,  where $\l_1,\l_2\in2^\bz$ are finite ascending ordered sets of $\bz$, including the empty set for the identity. The words of length 1 correspond to $a^\dagger_i=X_{(\{i\},\emptyset)}$, and $a_i=X_{(\emptyset,\{i\})}$, and for the words of length 2 of type $X_{(\{i\},\{j\})}$ one finds
\begin{align*}
&i\neq j\,\,\text{corresponds to}\,\, X_ {(\l_1,\l_2)}=a^\dagger_ia_j\,,\\
&i= j\,\,\text{corresponds to}\,\, X_{(\l_1,\l_2)}=a_ia^\dagger_i\,.
\end{align*}
The remaining cases give rise to $\l$-forms of length at least $2$ of the type 
$$
X_{(\l_1,\l_2)}=a^\dagger_{i_1}\cdots a^\dagger_{i_m}a_{j_1}\cdots a_{j_n}\,.
$$
with
$(\l_1,\l_2)=(\{i_1,\cdots i_m\},\{j_n,\cdots j_1\})$.\\

We now move on to recall how to associate with any permutation $\s\in \bp_{\mathbb{Z}}$ a 
$*$-linear map $T_\s$ acting on the dense subalgebra $\gam_0$, as was first done in \cite{CFG}.
A finite permutation $\s\in \bp_{\mathbb{Z}}$ is  said to be order preserving on a finite ordered subset 
$I:=\{i_1,\ldots, i_m\}\subset\bz$, with $i_1<\cdots < i_m$, if $\s(i_h)<\s(i_{h+1})$ for each $h=1,\ldots, m-1$.
The set of order-preserving permutations on $I$ is denoted by $OP(I)$.\\
For  elements $X_{(\l_1,\l_2)}$ and $\s\in\bp_\bz$, we define
\begin{equation*}
T_\s\big(X_{(\l_1,\l_2)}\big):=\left\{
\begin{array}{ll}
             X_{(\s(\l_1),\s(\l_2))} & \text{if}\,\, \s\in OP(\l_1)\cap OP(\l_2)\,, \\
             0 & \text{otherwise}\,.
           \end{array}
           \right.
\end{equation*}
The maps $T_\s$ are seen to extend by linearity to $*$-maps acting on $\gam_0$. However,
they are not positive maps, as shown 
in \cite{CFG}. To date it is not actually known whether the $T_\s$'s are bounded maps and thus extend by density
to the whole monotone $C^*$-algebra $\gam$.\\
\begin{rem}\label{THamel}
Note that for every $\s\in\bp_\bz$ the map $T_\s$ sends a $\pi$-form of the Hamel basis to a 
$\pi$-form of the Hamel basis and a $\lambda$-form to either a $\lambda$-form
of the same length or to $0$, where the second case may occur only if the length 
of the $\lambda$-form is at least $2$.
\end{rem}

Set $\ga:=\bm_2(\bc)\oplus\bc$ and  $a:=
 \left(
\begin{array}{ll}
0 & 1 \\
0 & 0
\end{array}
\right)\in\bm_2(\bc)
$.
Note that $a$ is a generator of $\bm_2(\bc)$ that satisfies the relation
$aa^*+a^*a= I_{\bm_2(\bc)}$.
Equalities  \eqref{comrul}, \eqref{comrul2} and
Lemma 5.4 in \cite{CFL} easily imply that the monotone
$C^*$-algebra coincides, up to isomorphism, with the quotient of
$\ast_\bz\ga$ modulo the monotone relations. More explicitly, the isomorphism 
is given by $a_j\leftrightarrow [i_j(a)]$, $j\in\bz$, where
$[i_j(a)]$ denotes the equivalence class of $i_j(a)$ in  $\ast_\bz\ga$.
That said, we can recall  what the process  corresponding to a given state $\f\in \cs(\gam)$
looks like. This is obtained by defining $\iota_j$  as 
\begin{equation}\label{iota}
\iota_j(a\oplus\lambda)=\pi_\f\bigg(a_j+\lambda\sum_{k<j }a_k^\dag a_k\bigg), \quad \lambda\in\bc
\end{equation}
where $\pi_\f:\gam\rightarrow\cb(\ch_\f)$ is the GNS representation of $\gam$ associated to $\varphi$.\\
In Section 5 of \cite{CFG}  monotone processese have been shown to be exchangeable if and only if the corresponding
monotone states $\varphi\in\cs(\gam)$ satisfy
$$\f\lceil_{\gam_0}\circ T_\s=\f\lceil_{\gam_0}\quad\textrm{for every}\,\s\in\bp_\bz\,.$$
We will consistently refer to such states as symmetric states.
We also recall that $\bz$ acts on $\gam$ through the shift automorphism 
$\alpha_\tau$ given by $\alpha_\tau(a_i)=a_{i+1}$, $i\in\bz$. Note that 
$\alpha_\tau$ comes from the corresponding action of $\bz$ at the level of
the free product $\ast_\bz\ga$. A state $\varphi$ on $\gam$
is called stationary if $\f\circ\a_\tau=\f$.

Denote now $\gb_0:=\text{span} \{X\in \gam_0: l(X)>0\}$ and $\gb$ its norm closure. Corollary 5.10 of \cite{CFL} gives that $\gam$ is the $C^*$-algebra obtained by adding the identity to $\gb$. Then every $Y\in\gam$ decomposes into the sum $Y:=X+cI$, where $X\in\gb$ and $c\in\mathbb{C}$.\\
The state at infinity $\om_\infty$ (see {\it e.g.} \cite{BR1})
is consequently well defined as
$$
\om_\infty(Y)=\om_\infty(X+cI):=c\,.
$$
In addition, the vacuum state $\om\in \cs(\gam)$ is given by
$$
\om(Y):=\langle Y \zeta,\zeta\rangle\,.
$$
We recall  that  symmetric states are all of the form 
$$
\f=\g\om+(1-\g)\om_{\infty}\, ,
$$ 
 for some $\g\in[0,1]$, see \cite[Proposition 5.1]{CFG}.\\

The next proposition provides a description of the operator system
$\mathfrak{M}_0^{\rm{ex}}:=\{X\in\mathfrak{M}_0: T_\sigma(X)=X\,\,\textrm{for every}\, \sigma\in \bp_\bz\}$, which
turns out to be trivial.

\begin{prop}\label{rotalgebra}
With the notation set above, we have $\mathfrak{M}_0^{\rm{ex}}=\mathbb{C}I$.
\end{prop}

\begin{proof}
We will actually prove more: for any $X\in\mathfrak{M}_0\setminus\mathbb{C}I$, there is a permutation
$\sigma\in \mathbb{P}_\mathbb{Z}$ such that $T_\s(X)\neq X$.
Now any such $X$ can be written as $X=\sum_{k\in F} \alpha_k X_{(\lambda_1^{(k)}, \lambda_2^{(k)})}$, where $F$ is a finite set,
and $\alpha_k$ are complex numbers.
Let $G\subset\mathbb{Z}$ be  the finite set obtained as the union $\bigcup_{k\in F} (\lambda_1^{(k)}\cup\lambda_2^{(k)})$.
It is enough to pick a permutation $\sigma\in\mathbb{P}_\mathbb{Z}$ such that $\sigma(G)\cap G=\emptyset$ and $\sigma\in OP(G)$ to have
$T_\s(X)\neq X$. Indeed, by construction for any $k\in F$ we have $T_\s(X_{(\lambda_1^{(k)}, \lambda_2^{(k)})})\neq X_{(\lambda_1^{(j)}, \lambda_2^{(j)})}$ with
$j\in F$, and the conclusion follows by linear independence as $T_\s$ sends $\lambda$-forms ($\pi$-forms) to
$\lambda$-forms ($\pi$-forms). All is left to do is exhibit such a permutation.
Denote by $m_G$ and $M_G$  the minimum and the maximum of $G$, respectively. If we define
$$
\sigma(h):=\left\{
\begin{array}{ll}
\tau^{M_g-m_g+1}(h) & h\in G \\
\tau^{-(M_g-m_g+1)}(h) & h\in \tau^{M_g-m_g+1}(G) \\
h & \textrm{otherwise}
\end{array}
\right.
$$
$\sigma$ is seen at once to enjoy the desired properties.
\end{proof}

We recall that
given a stochastic process $\big(\ga,\ch,\{\iota_j\}_{j\in J},\xi\big)$, its \emph{tail algebra} is the von Neumann subalgebra of
$\mathcal{B}(\ch)$ defined as
$$
\xi_{\f}^{\perp}:=\bigcap_{I\subset J,\,\, I\,\,\text{finite}}\bigg(\bigcup_{K\cap I=\emptyset,\,\, K\,\,\text{finite}}\bigg(\bigvee_{k\in K}\iota_k(\ga)\bigg)\bigg)''\,,
$$
where $\varphi$ is the corresponding state on $\ast_\bz\ga$.\\
It is one of our goals in the present work to come to a full description of the tail algebra of exchangeable processes. 
To this aim, we need to recall a simple general result from representation theory of $C^*$-algebras.
Although the result is likely to be well known, we include a precise statement and a sketched proof for want of a reference.\\
Let $\mathfrak{A}$ be a unital $C^*$-algebras and let $\om_1, \om_2$ be two states on $\mathfrak{A}$. If
 $\om_1$ and $\om_2$ are disjoint, namely the corresponding GNS representations are disjoint\footnote{Two representations $\pi_1$ and $\pi_2$ of $\mathfrak{A}$ are disjoint if
$(\pi_1, \pi_2)=\{0\}$, where $(\pi_1, \pi_2):=\{T\in\mathcal{B}(\ch_1, \ch_2): T\pi_1(a)=\pi_2(a)T,\,\,a\in\mathfrak{A}\}$.}, then the GNS representation of any non-trivial convex combination $\om=\gamma\om_1+(1-\gamma)\om_2$, with $0< \gamma< 1$, is obtained as a direct sum. More precisely, the GNS triple $(\mathcal{H}_\omega, \pi_\omega, \xi_\omega)$ can be
identified with
\begin{equation*}
\label{direct1}
\mathcal{H_\omega}=\mathcal{H}_{\om_1}\oplus\mathcal{H}_{\om_2}, \,\,\pi_\om=\pi_{\om_1}\oplus\pi_{\om_2}, \,\,
\xi_\om=\gamma^{\frac{1}{2}}\xi_{\om_1}\oplus(1-\gamma)^{\frac{1}{2}}\xi_{\om_2}
\end{equation*}
up to unitary equivalence. This is seen as follows.
To begin with, the equalities $\|\xi_\om\|^2=1$ and $\omega(a)=\langle\pi_\omega(a)\xi_\omega, \xi_\omega\rangle$, for any $a\in\mathfrak{A}$, can be checked by straightforward computation.
To conclude, we need only make sure that the vector $\xi_\om$ is cyclic for $\pi_\om$. This is where the disjointness of
$\om_1$ and $\om_2$ plays a role. Indeed, in this case it is well known (see \emph{e.g.} Theorem 10.3.5 in \cite{Kad}) that the von Neumann algebra generated by $\pi_\om$
decomposes into a direct sum, namely
\begin{equation*}
\label{direct2}
\pi_\om(\mathfrak{A})''=\pi_{\om_1}(\mathfrak{A})''\oplus\pi_{\om_2}(\mathfrak{A})''\,.
\end{equation*}
But then $\pi_\om(\mathfrak{A})''\xi_\om=\pi_{\om_1}(\mathfrak{A})''\xi_{\om_1}\oplus\pi_{\om_2}(\mathfrak{A})''\xi_{\om_2}$, which means $\xi_\om$ is cyclic for $\pi_\om(\mathfrak{A})''$, and so
it is cyclic for $\pi_\om(\mathfrak{A})$ as well, since the latter algebra is strongly dense in the former.\\

The general framework we outlined  above  applies in particular to all symmetric states because they are just a convex combination of the vacuum state
$\om$ and the state at infinity $\om_\infty$, which are both irreducible and inequivalent.
Therefore, if $\varphi=\gamma\om+(1-\gamma)\om_\infty$, with $0<\gamma<1$, then, up to unitary equivalence, we have
$$\mathcal{H}_\varphi=\cf_m\oplus\mathbb{C}, \,\, \pi_\varphi(X)=X\oplus\om_\infty(X)\,\,\textrm{for any}\, X\in\mathfrak{M}, \,\, \xi_\varphi=\gamma^\frac{1}{2}\zeta\oplus(1-\gamma)^\frac{1}{2}$$
 and $\pi_\varphi(\mathfrak{M})''=\pi_\om(\mathfrak{M})''\oplus\mathbb{C}=\cb(\cf_m)\oplus\bc$.\\
This allows us to thoroughly describe the tail algebra $\xi_\varphi^\perp$ when
$\varphi$ is a symmetric state. Denoting by $P_\zeta$ the orthogonal projection onto $\bc\zeta$ we have
the following

\begin{prop}\label{tailalgebra}
For any symmetric state $\varphi=\gamma\omega+(1-\gamma)\omega_\infty$, with $0\leq\gamma\leq 1$,
we have:

\begin{enumerate}
\item if $\gamma=0$, that is $\varphi=\omega_\infty$, then $\pi_{\omega_\infty}(\mathfrak{M})''=\mathbb{C}=\xi_{\omega_\infty}^\perp$;
\item if $\gamma=1$, that is $\varphi=\omega$, then $\pi_\omega(\mathfrak{M})''=\mathcal{B}(\cf_m)$ and
$\xi_\omega^\perp=\mathbb{C}P_\zeta\oplus\mathbb{C}P_\zeta^\perp$;
\item if $0<\gamma<1$,  then $\pi_\varphi(\mathfrak{M})''=\mathcal{B}(\cf_m)\oplus\mathbb{C}$ and $\xi_\varphi^\perp=(\mathbb{C}P_\zeta\oplus\mathbb{C}P_\zeta^\perp)\oplus\mathbb{C}$.
\end{enumerate}
\end{prop}

\begin{proof}
The case $\gamma=0$, \emph{i.e.} $\varphi=\omega_\infty$, is immediately dealt with as one has $\pi_{\omega_\infty}(\mathfrak{M})''=\pi_{\omega_\infty}(\mathfrak{M})=\mathbb{C}=\xi_{\omega_\infty}^\perp$.\\

If $\gamma=1$, \emph{i.e.} $\varphi=\omega$, the GNS representation $\pi_\varphi$ is nothing
but the Fock representation of the monotone annihilators, which is irreducible, see \cite[Proposition 5.9]{CFL}. Accordingly,
we have $\pi_\omega(\mathfrak{M})''=\mathcal{B}(\cf_m)$. Taking into account \eqref{iota},
the corresponding process is
given by:
\begin{equation}
\label{iota1}
\iota_j(a\oplus\beta)=a_j+\beta \sum_{k<j}a^\dag_ka_k, \quad \beta\in\bc.
\end{equation}

As $\{a_k\}_{k\in\bz}$ have orthogonal ranges, the last term in \eqref{iota1} is merely the projection onto the subspace
$$
\overline{{\rm span}}\{e_{i_1}\otimes\ldots\otimes e_{i_m}: m\geq 1,
j>i_1<i_2<\ldots<i_m\}
$$
of $\cf_m$, and belongs to $\gam$ by virtue of \eqref{comrul2}. 
Since in our case the tail algebra is given by
$$
\xi_\omega^\perp=\bigcap_{n\in\bn}\left(\bigvee_{|j|>n}\iota_j\big(\bm_2(\bc)\oplus \bc\big)\right)''\,,
$$
from \eqref{iota1} it easily follows that $\sum_{k\in\bz}a^\dag_ka_k$ belongs to the tail algebra, which means 
$P_\zeta$  sits in it as well because of the equality $\sum_{k\in\bz}a^\dag_ka_k=I-P_\zeta$.\\
Thus,
we only need to prove the other inclusion, namely $\xi_\omega^\perp\subset\mathbb{C}P_\zeta\oplus\mathbb{C}P_\zeta^\perp$.
To this end, let us set $\mathcal{R}_n:=\left(\bigvee_{|j|>n}\iota_j\left(\mathbb{M}_2(\mathbb{C}\right)\oplus\mathbb{C})\right)''$ for every $n$. By its very definition, it is clear
that $\mathcal{R}_n$ is the strong closure of $\gam_n$, the unital $C^*$-algebra generated by
$\{a_j:|j|>n\}$.\\
Now if we set
$$
\cf_m^{(n)}:=\overline{\rm span}\{e_{i_1}\otimes\ldots\otimes e_{i_m}: m\geq 1,\,\, n\geq i_1<\ldots < i_m\}\,,
$$
it is a matter of trivial computations to ascertain that $\cf_m^{(n)}$ (along with its orthogonal complement $(\cf_m^{(n)})^\perp$) is invariant for $\gam_n$.
But then any operator sitting in the von Neumann algebra $\mathcal{R}_n$, too, must preserve
both $\cf_m^{(n)}$ and $(\cf_m^{(n)})^\perp$. Phrased differently, for every natural number $n$ we have proved the inclusion
$\mathcal{R}_n\subset\mathcal{G}_n$, where $\mathcal{G}_n\subset\mathcal{B}(\cf_m)$ is
the von Neumann algebra of all block diagonal operators w.r.t. the decomposition of the monotone Fock
space $\cf_m$ into the direct sum $\cf_m^{(n)}\oplus (\cf_m^{(n)})^\perp$. The thesis will then be proved if we make sure that
$\bigcap_n\mathcal{G}_n$ reduces to $\mathbb{C}P_\zeta\oplus\mathbb{C}P_\zeta^\perp$. This can be seen
as follows. Any bounded linear operator $T$ lies in the intersection $\bigcap_n\mathcal{G}_n$ if and only if, for every
$n\in\bn$, it preserves both $\cf_m^{(n)}$ and $(\cf_m^{(n)})^\perp$. As $\cf_m^{(n)}\subset \cf_m^{(n+1)}$ (and $(\cf_m^{(n)})^\perp\supset (\cf_m^{(n+1)})^\perp$),
$T$ must preserve $\bigcup_n \cf_m^{(n)}$ (and $\bigcap_n (\cf_m^{(n)})^\perp$), and the conclusion is thus arrived at because
$\bigcup_n \cf_m^{(n)}=(\mathbb{C}\zeta)^\perp$.

Finally, the general case of a proper convex combination $\varphi$ of $\omega$ and $\omega_\infty$ can be handled
easily by taking into account the decomposition of the GNS representation $\pi_\varphi$ we discussed above.
\end{proof}

\begin{rem}\label{stationary}
It is a classical result  due to Hewitt and Savage, \cite{HS}, that exchangeable and tail algebra of a sequence of 
exchangeable random variables are the same. Furthermore, under the same hypothesis exchangeable and  stationary algebra
also coincide by virtue of a theorem of Olshen, see \cite{Ols}.
As soon as quantum processes are looked at, though, these classical results
may cease to hold. For instance, Boolean processes are certainly a case in point, for in \cite{Fbo}
symmetric and stationary algebra of any symmetric Boolean state are shown to differ. \\
Monotone processes are worse behaved. Indeed, the exchangeable algebra may even be not defined
because it is still not  known whether the maps $T_\sigma$ extend to the $C^*$-algebra $\gam$.
To make matters worse, even if  the $T_\sigma$'s did extend to $\gam$, it would not be clear if
they may be extended to the von Neumann algebra $\pi_\varphi(\gam)''$ as well.  
Nevertheless, what we do know is  that the tail algebra
$\xi_\om^\perp=\mathbb{C}P_\zeta+\mathbb{C}P_\zeta^\perp\cong \mathbb{C}\oplus\mathbb{C}$ is strictly included in  the stationary algebra $\pi_\om(\mathfrak{M})''_\bz$. Indeed, by definition, $\pi_\om(\mathfrak{M})''_\bz$ is given by all operators
$T\in\mathcal{B}(\cf_m)$  that commute with the unitary $U$ implementing $\a_\tau$. This acts on the canonical basis of  $\cf_m$ as
\begin{eqnarray*}
&&U\zeta:=\zeta\\
&&U e_{i_1}\otimes e_{i_2}\ldots \otimes e_{i_n}:= e_{i_1+1}\otimes e_{i_2+1}\ldots \otimes e_{i_n+1},\quad i_1<i_2<\cdots<i_n
\end{eqnarray*}
Therefore, the unitary $U$ does sit in $\pi_\om(\mathfrak{M})''_\bz$ while not belonging to $\xi_\om^\perp$.\\
Furthermore, it may be worth noting that although the stationary $C^*$-subalgebra $\gam_\bz:=\{X\in\gam: \a_\tau(X)=X\}$ is trivial \cite[Proposition 5.11]{CFL},
the stationary von Neumann algebra $\pi_\om(\gam)''_\bz$ is not even
commutative. This can be seen by observing that the operator $S$ defined by
\begin{eqnarray*}
&&S\zeta:=\zeta\\
&&Se_{i_1}\otimes e_{i_2}\ldots \otimes e_{i_n}:= e_{i_1}\otimes e_{i_2}\ldots \otimes e_{i_n+1}, \quad i_1<i_2<\cdots<i_n
\end{eqnarray*}
commutes with $U$, but $S$ and its adjoint do not commute since $S$ is a proper isometry ($e_1\otimes e_2$ is not in the range of $S$), to wit $S^*S=I$ but $SS^*<I$.
\end{rem}

The following is a straightforward adaptation to the monotone case of the arguments employed to arrive at
Formula $(4)$  in \cite{Fbo}. Denote by $\tilde{\omega}$  the unique normal extension of $\omega$ to $\mathcal{B}(\cf_m)$, {\it i.e.}
$\tilde{\omega}(T)=\langle T\zeta, \zeta\rangle$, for any $T\in\mathcal{B}(\cf_m)$.

\begin{prop}
For any $\varphi=\gamma\omega+(1-\gamma)\omega_\infty$, with $0\leq\gamma\leq 1$,   there exist conditional expectations $E$ from $\pi_\f(\gam)''$ onto $\xi_{\f}^{\perp}$.
In addition:
\begin{enumerate}
\item If $\gamma=0$, \emph{i.e.} $\varphi=\omega_\infty$, the only such conditional expectation is
$E={\rm id}_{\bc}:\bc\rightarrow \bc$.

\medskip

\item If $\gamma=1$, \emph{i.e.} $\varphi=\omega$ any such conditional expectation $E$ is of the form $E_\psi:\mathcal{B}(\cf_m)\rightarrow \bc P_\zeta \oplus \bc P_\zeta^\perp$, where
\begin{equation*}
\label{effe2}
E_\psi(X)=\widetilde{\om}(X)P_{\zeta}+\psi(P_{\zeta}^{\perp}X P_{\zeta}^{\perp})P_{\zeta}^{\perp}\,,\quad X\in\mathcal{B}(\cf_m)
\end{equation*}
and $\psi$ is a state on $\mathcal{B}(\cf_m)$ such that $\psi(P_\zeta^\perp)=1$.\\
Moreover, any $E_\psi$ is $\widetilde{\omega}$-invariant.

\medskip

\item If $0<\gamma<1$, any such conditional expectation $E$ is of the form $E_\psi\oplus {\rm id}_\bc:\mathcal{B}(\cf_m) \oplus \bc \rightarrow (\bc P_\zeta \oplus \bc P_\zeta^\perp)\oplus \bc$. 

\noindent Moreover, any $E$ of form above preserves the vector state $\langle\cdot\xi_\f,\xi_\f\rangle$.
\end{enumerate}
\end{prop}

In what follows we will have to make use of a technical result that guarantees the existence of states with suitable properties.

\begin{lem}\label{suitstate}
There exist states on $\mathcal{B}(\cf_m)$ that vanish on both $\mathfrak{B}$ and
$\mathcal{K}(\cf_m)$, the compact operators on the monotone Fock  space $\cf_m$.
\end{lem}

\begin{proof}
We claim that $\|A+K+\lambda I\|\geq |\lambda|$ for any $A\in\mathfrak{B}$, $K\in\mathcal{K}(\cf_m)$, and
$\lambda\in\mathbb{C}$.
Note that $\mathfrak{B}+\mathcal{K}(\cf_m)$ is a (possibly not uniformly closed) $*$-subalgebra of $\mathcal{B}(\cf_m)$. Thanks to the claimed inequality, the linear functional
$\om_\infty$ defined on the unital
$*$-algebra $\mathfrak{B}+\mathcal{K}(\cf_m)+\mathbb{C}I$ by $\om_\infty(A+K+\lambda I):=\lambda$
is bounded and positive as its norm, which is equal to $1$, is attained on the identity $I$. By virtue of the Segal extension theorem (see \emph{e.g.} \cite{Kad1}, Theorem 4.3.13), it can be extended to a state $\widetilde{\om}_\infty$
defined on the whole $\mathcal{B}(\cf_m)$ which by definition vanishes on both
$\mathfrak{B}$ and $\mathcal{K}(\cf_m)$. \\
All is left to do is prove the  claimed inequality, which can be done by a variation of the proof of Proposition 5.8 in \cite{CFL}.
Thanks to a standard approximation argument, it is enough to ascertain  that the inequality holds true
for $A$ in $\mathfrak{B}_0$. Now any such $A$ is a sum $\sum_{i\in F } \beta_i X_i$, where $F$ is a finite set and
$X_i$ is either a $\lambda$-form or a $\pi$-form with $l(X_i)>0$, namely
$X_i = Y_ia_{j_i}^\natural$, for any $i\in F$,  where $a_{j_i}^\natural$ is either $a_{j_i}$ or $a_{j_i}^\dag$. If now $n\in\bz$ is less than $\min\{j_i: i\in F\}$, we have
$X_i e_n=0$ for every $i\in F$. But then
\begin{align*}
\|A+K+\lambda I\|&\geq\|(A+K+\lambda I)e_n\|\\
&=\|Ke_n+\lambda e_n\|\\
&\geq |\lambda|- \|Ke_n\|
\end{align*}
Taking the limit of the above inequality as $n$ tends to $-\infty$, we eventually get the claim because
$\|Ke_n\|$ tends to $0$ by compactness (a compact operator sends any orthonormal system to a sequence converging to $0$ in norm).
\end{proof}

By a slight abuse of notation we will denote by $\widetilde{\om}_\infty$ any of the states whose existence
has been established above.\\

Our next goal is to prove that all monotone symmetric states give rise to
\emph{conditionally independent and identically distributed} (with respect to the tail algebra) stochastic processes,
and conversely that any conditonally independent and identically distributed process is exchangeable.
For the reader's convenience, we recall the necessary definitions.
\begin{defin}\label{iid}
A stochastic process $(\ga, \ch, \{\iota_j\}_{j\in J}, \xi)$ with corresponding
state $\varphi$ is \emph{conditionally independent and identically distributed} w.r.t. the tail algebra if there exists a conditional expectation\footnote{Given an inclusion $\mathfrak{B}\subset\mathfrak{A}$ of $C^*$-algebras, a conditional expectation $E: \mathfrak{A}\rightarrow\mathfrak{B}$ is a positive linear map such that $E(b)=b$, for any $b\in\mathfrak{B}$,
and $E(b_1ab_2)=b_1E(a)b_2$, for any $b_1, b_2\in\mathfrak{B}$, $a\in\ga$.}
$$
E_{\f}:\bigvee_{j\in J} \iota_{j}(\ga)\rightarrow \xi_{\f}^{\perp}
$$
preserving the vector state $\langle \cdot \xi, \xi\rangle$, such that

\bigskip

(i) $\langle XY\xi,\xi\rangle=\langle E_{\f}(X)E_{\f}(Y)\xi,\xi\rangle$, for each finite subsets $I,K\subset J$, $I\cap K=\emptyset$, and
$$
X\in \bigg(\bigvee_{i\in I} \iota_i(\ga)\bigg)\bigvee \xi_{\f}^{\perp}\,, \quad Y\in \bigg(\bigvee_{k\in K} \iota_k(\ga)\bigg)\bigvee \xi_{\f}^{\perp}\,.
$$

(ii) $E_{\f}(\iota_i(a))=E_{\f}(\iota_k(a))$, for each $i,k\in J$, and $a\in\ga$.\\
\end{defin}

\begin{thm}\label{definetti}
For a state $\varphi$ on the monotone $C^*$-algebra  $\gam$ the following conditions are equivalent:
\begin{itemize}
\item[(i)] $\varphi$ is symmetric;
\item[(ii)] the stochastic process corresponding to $\varphi$ is conditionally independent and identically distributed w.r.t. the tail algebra.
\end{itemize}
\end{thm}

\begin{proof}
We start by proving that $(i)$ implies $(ii)$.
As any symmetric state is a convex combination of the vacuum state $\omega$ and  the state at infinity $\om_\infty$, it is enough to consider
the case $\varphi=\om$, as is done in \cite{Fbo}. This amounts to exhibiting a conditional expectation
$E:\mathcal{B}(\cf_m)\rightarrow\xi_\omega^\perp$ such that:
\begin{enumerate}
\item  Given any two finite sets $H, J\subset\mathbb{Z}$ with $H\cap J=\emptyset$, for any
$X\in \left(\bigvee_{i\in H} \iota_i(\bm_2(\mathbb{C})\oplus\mathbb{C})\right)\bigvee\xi_\om^\perp$ and $Y\in \left(\bigvee_{j\in J} \iota_j(\bm_2(\mathbb{C})\oplus\mathbb{C})\right)\bigvee\xi_\om^\perp$ we have:
$$\langle XY\zeta, \zeta \rangle=\langle E(X)E(Y)\zeta, \zeta\rangle$$
\item $E(\iota_j(X))=E(\iota_k(X))$ for any $X\in \bm_2(\mathbb{C})\oplus\mathbb{C}$, $j, k\in\bz$.
\end{enumerate}
We will show that the two conditions above are both satisfied if we take $E=E_{\widetilde{\om}_\infty}$,  namely
$E_{\widetilde{\om}_\infty}(X)=\widetilde{\om}(X)P_{\zeta}+\widetilde{\om}_\infty(P_{\zeta}^{\perp}X P_{\zeta}^{\perp})P_{\zeta}^{\perp}$, for any $X\in\mathcal{B}(\cf_m)$, where $\widetilde{\om}_\infty$ is a state as in the statement of Lemma \ref{suitstate}.\\
We start by checking the second condition. Given $a=(\alpha_{i,j})\in \bm_2(\mathbb{C})$ and $\beta\in\mathbb{C}$, for any
$l\in\bz$ we have
$$\iota_l(a\oplus\beta)=\alpha_{1,1} a_la^\dag_l +\alpha_{1, 2}a_l +\alpha_{2, 1}a^\dag_l+ \alpha_{2, 2}a^\dag_l a_l+\beta\sum_{k<l}a^\dag_k a_k.$$
But then $E_{\widetilde{\om}_\infty}(\iota_l(a\oplus\beta))=\widetilde{\omega}(\iota_l(a\oplus\beta))P_\zeta+\widetilde{\om}_\infty(P_\zeta^\perp\iota_l(a\oplus\beta) P_\zeta^\perp)P_\zeta^\perp$.
The first summand is easily seen not to depend on $l$ since
$$\om(\iota_l(a\oplus\beta))=\alpha_{1,1}.$$
The second summand does not depend on $l$ either  because
$$\widetilde{\om}_\infty(P_\zeta^\perp\iota_l(a\oplus\beta) P_\zeta^\perp)=\beta$$
 thanks to
the properties of $\widetilde{\om}_\infty$ as stated in Lemma \ref{suitstate}, the fact that $\mathcal{K}(\cf_m)$ is a two-sided ideal in $\mathcal{B}(\cf_m)$, and the commutation rules \eqref{comrul2}.\\

We now move on to check the first condition. We start by claiming that any $X$ in $\left(\bigvee_{i\in H} \iota_i(\bm_2(\mathbb{C})\oplus\mathbb{C})\right)\bigvee\xi_\om^\perp$  can be written as a sum
$$X=\lambda P_\zeta + T_0(X)+ B$$
where  $\lambda\in\mathbb{C}$, $T_0(X)$ is an operator such that either $T_0(X)\zeta=0$ or
$T_0(X)\zeta$ is a linear combination of tensors whose indices are all in $H$ (in particular,
$\widetilde{\om}(T_0(X))=(T_0(X)\zeta, \zeta)=0$), and $B$ is a linear combination of the type
$$B=\sum_{j_1\in F_1} \gamma_{j_1}B_{j_1}+\sum_{j_2\in F_2}\gamma'_{j_2}B_{j_2}'$$
where $F_1, F_2$ are finite sets, the $B_{j_1}$'s, $j_1\in F_1$, are $\l$-forms whose indices are chosen from $H$ and
the ${B}_{j_2}'$'s, $j_2\in F_2$, are $\pi$-forms whose indices are chosen from $H$ as well. In addition, we do not harm generality if we also assume that the $\lambda$-forms and the $\pi$-forms above are all taken from
the Hamel basis of $\gam_0$.

The above decomposition is a straightforward consequence of the following facts:\\
{\it a)} for any $a, a'\in\bm_2(\bc)$ and $\beta, \beta'\in\bc$, $\iota_j(a\oplus\beta)\iota_k(a'\oplus\beta')$ is a linear combination of $\lambda$-forms and $\pi$-forms whose
indices are  $i$ and $k$, as follows from Lemma 5.4 in \cite{CFL} and \eqref{comrul}--\eqref{comrul2};\\
{\it b)} For $a=(\alpha_{i,j})\in\bm_2(\bc)$, $\iota_j(a\oplus\beta)P_\zeta= \alpha_{1,1}P_\zeta+\alpha_{2,1}a^\dag_j P_\zeta$.

By the same token, any $Y$ in  $\left(\bigvee_{j\in J} \iota_j(\bm_2(\mathbb{C})\oplus\mathbb{C})\right)\bigvee\xi_\om^\perp$
can be written as a sum $Y=\mu P_\zeta + T_0(Y)+ C$,
where $\mu\in\mathbb{C}$, while $T_0(Y)$ enjoys similar properties to $T_0(X)$ and $C$ is given by
$$C=\sum_{h_1\in G_1} \beta_{h_1}C_{h_1}+\sum_{h_2\in G_2}\beta'_{h_2}C_{h_2}'$$
where $G_1, G_2$ are finite sets, the $C_{h_1}$'s, $h_1\in G_1$, are $\l$-forms whose indices are chosen from $J$ and
the ${C}_{h_2}'$'s, $h_2\in G_2$, are $\pi$-forms whose indices are chosen from $J$ as well.

This enables us to write the product $XY$ in a form that is suited to the computations we need to make. Indeed, we have:
\begin{equation}\label{prod}
XY=\l\mu P_\zeta +\l P_\zeta C+\mu BP_\zeta+{T_0}(XY)+ BC
\end{equation}
where $T_0(XY)$ is an operator such that $\widetilde{\om}(T_0(XY))=(T_0(XY)\zeta, \zeta)=0$.
In more detail, $T_0(XY)$ accounts for the sum of the following terms arising from the product $XY$:
$$\lambda P_\zeta T_0(Y),\,\,\mu T_0(X)P_\zeta,\,\,T_0(X)T_o(Y),\,\,T_0(X)C,\,BT_0(Y)$$
Thanks to the properties of $T_0(X)$ and $T_0(Y)$, the assumption
$H\cap J=\emptyset$ readily implies that $\widetilde{\om}$ vanishes on each of the above terms.
Now we have $$E_{\widetilde{\om}_\infty}(X)=(\lambda+\widetilde{\om}(B))P_\zeta+\widetilde{\om}_\infty(P_\zeta^\perp X P_\zeta^\perp)P_\zeta^\perp$$ and 
$$E_{\widetilde{\om}_\infty}(Y)=(\mu+\widetilde{\om}(C))P_\zeta+\widetilde{\om}_\infty(P_\zeta^\perp Y P_\zeta^\perp)P_\zeta^\perp\,,$$ which means
\begin{equation*}
E_{\widetilde{\om}_\infty}(X)E_{\widetilde{\om}_\infty}(Y)=
(\lambda+\widetilde{\om}(B))(\mu+\widetilde{\om}(C))P_\zeta+
\widetilde{\om}_\infty(P_\zeta^\perp X P_\zeta^\perp)\widetilde{\om}_\infty(P_\zeta^\perp Y P_\zeta^\perp)P_\zeta^\perp
\end{equation*}
and so the equality
$$\langle E_{\widetilde{\om}_\infty}(X)E_{\widetilde{\om}_\infty}(Y)\zeta, \zeta\rangle=\lambda\mu+\lambda\widetilde{\om}(C)+\mu\widetilde{\om}(B)+\widetilde{\om}(B)\widetilde{\om}(C)$$
is got to. In addition, thanks to Equality (\ref{prod}) we have
$$\langle XY\zeta, \zeta\rangle=\l\mu+\l\widetilde{\om}(C)+\mu\widetilde{\om}(B)+\widetilde{\om}(BC).$$
The thesis will then follow if we prove that $\widetilde{\om}(BC)=\widetilde{\om}(B)\widetilde{\om}(C)$.
Arguing as in the proof of \cite[Theorem 3.4]{CFG}, one finds
\begin{align*}
\widetilde{\om}(BC)&=\widetilde{\om}\left(\sum_{j_2\in F_2, h_2\in G_2}\gamma'_{j_2}\beta'_{h_2}B'_{j_2}C'_{h_2}\right)\\
&=\sum_{j_2\in F_2, h_2\in G_2}\gamma'_{j_2}\beta'_{h_2}=\widetilde{\om}(B)\widetilde{\om}(C)
\end{align*}
in that $\widetilde{\om}=\langle \cdot\zeta , \zeta\rangle$ vanishes on all $\lambda$-forms and is $1$ on all
$\pi$-forms in the Hamel basis.\\

 We now move on to the implication $(ii)\Rightarrow (i)$.
Our strategy is to show that if a state $\varphi$ satisfies  the conditions in Definition \ref{iid}, then its restriction
to $\gam_0$ is invariant
for all maps $T_\s$, $\sigma\in\bp_\bz$. To this aim, in light of Remark \ref{THamel} it is enough to show that any such $\varphi$
vanishes on all $\lambda$-forms of the Hamel basis of length greater than $1$ and is constant on all $\pi$-forms of the same basis. For the computations we need to make it is useful to recall that for every $j\in\bz$ one has
$\iota_j(a)=\pi_\varphi(a_j)$, where $a\in\bm_2(\bc)$ is the matrix
$
 \left(
\begin{array}{ll}
0 & 1 \\
0 & 0
\end{array}
\right)
$.\\
The first property we verify is that $\varphi(a_ja^\dag_j)=\varphi(a_ha^\dag_h)$ for any $j, h\in\bz$.
This is seen as follows:
\begin{align*}
\varphi(a_ja^\dag_j)&= \langle E_\varphi (\pi_\varphi(a_j a^\dag_j))\xi_\f, \xi_\f \rangle=
\langle E_\varphi(\iota_j(aa^*))\xi_\f, \xi_\f \rangle\\
&= \langle E_\varphi(\iota_h(aa^*))\xi_\f, \xi_\f \rangle= \varphi(a_h a^\dag_h)
\end{align*}
where in second-last equality we exploited the assumption that the process is identically distributed.\\
Let now $X_{(\l_1,\l_2)}$ 
be a $\lambda$-form, where $\l_1:=\{i_1, i_2, \ldots, i_m\}$, $\l_2:=\{j_n, \ldots, j_2, j_1\}$, and finally $l(X_{(\l_1,\l_2)})> 1$. We want to make sure that $\varphi(X_{(\l_1,\l_2)})=0$. Fix $r\in\bz$ greater than
$\max \{i_m, j_1\}$. There are two cases to deal with depending on whether $\lambda_1$ is empty or not.\\
If $\lambda_1=\emptyset$, using \cite[Lemma 1]{Fbo} we find
\begin{align*}
\varphi(X_{(\emptyset,\l_2)})&=\varphi(a_{j_1}a_{j_2}\cdots a_{j_n})\\
&=\langle E_\f(\pi_\f(a_{j_1}a_{j_2}\cdots a_{j_n}))  \xi_\f, \xi_\f\rangle\\
&=\langle E_\f(\iota_{j_1}(a)\iota_{j_2}(a)\cdots\iota_{j_n}(a))  \xi_\f, \xi_\f\rangle\\
&=\langle E_\f(\iota_{j_1}(a)) E_\f(\iota_{j_2}(a))\cdots E_\f(\iota_{j_n}(a))  \xi_\f, \xi_\f\rangle\\
&=\langle E_\f(\iota_{j_1}(a)) E_\f(\iota_{j_2}(a))\cdots E_\f(\iota_{j_r}(a))  \xi_\f, \xi_\f\rangle\\
&=\langle E_\f(\pi_\f(a_{j_1}a_{j_2}\cdots a_{j_{n-1}}a_{j_r}))  \xi_\f, \xi_\f\rangle=0
\end{align*}
since $a_{j_{n-1}}a_{j_r}=0$ by \eqref{comrul}.\\
If $\lambda_1\neq\emptyset$, we pick $l\in\bz$ strictly less than $\min\{i_1, j_n\}$.
By Lemma 5.4 in \cite{CFL} we have
\begin{align*}
\varphi(X_{(\l_1,\l_2)})&=\varphi(a_la_l^\dag X_{(\l_1,\l_2)})\\
&=\big\langle E_\varphi\big(\iota_l(aa^*)\iota_{i_1}(a^*)\cdots \iota_{i_m}(a^*) \iota_{j_1}(a)\cdots \iota_{j_n}(a)\big) \xi_\f,\xi_\f \big\rangle\\
&=\big\langle E_\varphi\big(\iota_l(aa^*)\big) E_\varphi\big(\iota_{i_1}(a^*)\cdots \iota_{i_m}(a^*) \iota_{j_1}(a)\cdots \iota_{j_n}(a)\big) \xi_\f,\xi_\f \big\rangle\\
&=\big\langle E_\varphi\big(\iota_r(aa^*)\big) E_\varphi\big(\iota_{i_1}(a^*)\cdots \iota_{i_m}(a^*) \iota_{j_1}(a)\cdots \iota_{j_n}(a)\big) \xi_\f,\xi_\f \big\rangle\\
&=\big\langle \pi_\f(a_r a^\dag_r a^\dag_{i_1}\cdots a^\dag_{i_m}a_{j_1}\cdots a_{j_n})\xi_\f, \xi_\f \big\rangle=0
\end{align*}
since $a^\dag_r a^\dag_{i_1}=0$ by \eqref{comrul}.

\end{proof}

\section{The case of $q$-deformed processes}
\label{qdeformed}
The present section contains a result on the structure of the tail algebra
of the vacuum state on the $C^*$-algebra generated by operators satisfying the so-called $q$-deformed commutation
rules. Furthermore, it also includes some remarks about the
stationary and exchangeable algebras. 
We start by recalling how the $C^*$-algebra alluded to above can be obtained concretely.
This is by definition the $C^*$-algebra generated by creation and annihilation operators on the $q$-deformed Fock space \cite{BS}, whose
construction we next sketch for convenience.

To this aim, we take $-1<q<1$ and fix a Hilbert space $\ch$. The $q$--deformed Fock space $\G_q(\ch)$ is the completion of
the algebraic linear span of the vacuum vector $\z_q$, together with
vectors
$$
f_1\otimes\cdots\otimes f_n\,,\quad
f_j\in\ch\,,j=1,\dots,n\,,n=1,2,\dots
$$
with respect to the $q$--deformed inner product
\begin{equation*}
\langle f_1\otimes\cdots\otimes
f_m\,,g_1\otimes\cdots\otimes g_n\rangle_q
:=\d_{n,m}\sum_{\pi\in\bp_n}q^{i(\pi)}\langle
f_1\,,g_{\pi(1)}\rangle_\ch\cdots\langle f_n\,,g_{\pi(n)}\rangle_\ch\,,
\end{equation*}
where $i(\pi)$ is the
number of inversions of $\pi\in\bp_n$.
Fix $f,f_1,\ldots,f_n\in\ch$. Define the creator $l^\dagger(f)$ as
$$
l^\dagger(f)\z_q=f\,,\quad l^\dagger(f)f_1\otimes\cdots\otimes f_n=f\otimes
f_1\otimes\cdots\otimes f_n\,,
$$
and the annihilator $l(f)$ as
\begin{align}
\begin{split}
\label{qann}
&l(f)\z_q=0\,,\quad
l(f)(f_1\otimes\cdots\otimes f_n)\\
=&\sum_{k=1}^nq^{k-1}\langle
f_k,f\rangle_\ch f_1\otimes\cdots f_{k-1}\otimes
f_{k+1}\otimes\cdots\otimes f_n\,.
\end{split}
\end{align}
$l^\dagger(f)$ and $l(f)$ are mutually adjoint with respect to the $q$--deformed inner
product, and are continuous since for each $f\in\ch$
$$
\|l(f)\|=\|l^\dag(f)\|\leq \left\{\begin{array}{ll}
                                                                      \frac{1}{\sqrt{1-q}}\|f\| & \text{if $0\leq q<1$} \\
                                                                      \|f\| & \text{if $-1 <q\leq 0$}
                                                                    \end{array}
                                                                    \right.\,.
$$

 In addition, they are the Fock representation of the  $q$-commutation relations \cite{BS}, that is satisfy
\begin{equation}
\label{qcom}
l(f)l^\dagger(g)-ql^\dagger(g)l(f)=\langle
g,f\rangle_{\ch}I\,,\qquad f,g\in\ch\,.
\end{equation}
The limit cases are the Canonical Commutation Relations (Bosons) when $q=1$, and the
Canonical Anticommutation Relations (Fermions) for $q=-1$. The case $q=0$ corresponds to the free group reduced $C^*$--algebra \cite{VDN}.

Let us now take $\ch=\ell^2(\bz)$ with the canonical orthonormal basis $\{e_j\}_{j\in\bz}$, and on $\G_q(\ell^2(\bz))$ denote $l_j:=l(e_j)$, $j\in\bz$. The concrete $C^*$--algebra $\gar_q$ and its subalgebra $\gg_q$, acting on $\G_q(\ell^2(\bz))$ are the unital
$C^*$--algebras generated by the annihilators $\{l_j\mid j\in\bz\}$, and the self-adjoint part of
annihilators $s_j:=l_j+l_j^\dagger$, $j\in\bz$, respectively. Notice that $\gar_0$ is nothing but the Cuntz algebra $\co_\infty$ \cite[p. 6]{VDN}.\\
Denote by $\om_q:=\langle\cdot \z_q,\z_q\rangle$ the Fock vacuum vector state. For any word $X:=l^{\natural}_{i_1}l^{\natural}_{i_2}\cdots l^{\natural}_{i_n}$, where $\natural\in\{1,\dag\}$, $i_1,i_2,\ldots, i_n\in\bz$, $n\in\bn$, we say that $X$ is Wick-ordered if all the creators are on the left of the annihilators.
Note that thanks \eqref{qcom} to any word in the generators of $\mathfrak{R}_q$ can be rewritten as the sum of a multiple
of the identity with a finite linear combination of Wick-ordered words.\\

Our aim is to define a stochastic process on $\gar_q$. To do so, we first need to define our sample algebra $\ga$.
This will be the universal
$C^*$-algebra generated by an element $a$ such that $aa^*-qa^*a=\idd$.
Note that such algebra exists as the $C^*$-maximal (semi)-norm of $a$ is bounded from above by
$\frac{1}{\sqrt{1-|q|}}$ for any $-1<q<1$.
For any $j\in\bz$, by universality of $\ga$ there exists a unique $*$-homomorphism $\iota_j:\ga\rightarrow \mathcal{B}(\G_q(\ch))$ such that $\iota_j(a)=l_j$. In this way we obtain the quantum stochastic process $(\ga,  \G_q(\ell^2(\bz)), \{\iota_j\}_{j\in\bz}, \z_q)$.\\

The groups $\bz$ and $\bp_\bz$ both act on $\mathfrak{R}_q$ naturally. The ergodic properties of the corresponding
dynamical systems, namely shift invariance and exchangeability, have been addressed in \cite{DyF} and \cite{CrFid}, respectively.
The unital semigroup $\bl_\bz$, too,  acts on $\mathfrak{R}_q$ by $*$-endomorphisms $\L_g$, $g\in\bl_\bz$, given
by $\L_g(l_i):=l_{g(i)}$, $i\in\bz$. Thus, any  spreadable stochastic process corresponds to a state $\f$ such that $\f\circ\L_g=\f$, $g\in\bl_\bz$.
For completeness' sake, we also recall that from \cite[Theorem 3.3 and Corollary 3.4]{DyF} and \cite[Proposition 6.2]{CrFid}
the chain of equalities 
$$
\cs_{\bp_\bz}(\gar_q)=\cs_{\bl_\bz}(\gar_q)=\cs_{\bz}(\gar_q)=\{\om_q\}\,.
$$
and
$$
\cs_{\bp_\bz}(\gg_q)=\cs_{\bl_\bz}(\gg_q)=\cs_{\bz}(\gg_q)=\{\om_q\}.
$$
hold.

\begin{prop}\label{trivialq}
The tail algebra $\xi_{\om_q}^{\perp}$ is trivial.
\end{prop}
\begin{proof}
Let us define the decreasing sequence of von Neumann algebras $\mathcal{R}_n:=W^*(\{l_j: |j|\geq n\})\subset\cb( \G_q(\ell^2(\bz)))$. The thesis amounts to showing that $\bigcap_{n=1}^\infty \mathcal{R}_n=\bc I$.
To this end, let us also define the decreasing sequence of $*$-algebras $\mathcal{A}_n\subset\mathcal{R}_n$, where
$\mathcal{A}_n$ is the involutive subalgebra of $\cb( \G_q(\ell^2(\bz)))$ generated by the infinite set
$\{l_j: |j|\geq n\}$. Note that by definition  $\mathcal{A}_n''=\mathcal{R}_n$ for every $n\in\bn$.
We consider the set  in $\G_q(\ell^2(\bz))$ made up by the vacuum vector $\zeta_q$ and the countable collection of vectors  of the form
$e_{j_1}\otimes\cdots\otimes e_{j_k}$, with $k\in\bn$ and $j_1, j_2,\ldots, j_k\in\bz$.
Henceforth we will denote this set by $\{\eta_i: i\in\bn\}$.
Let now $T$ be an operator sitting in the intersection $\bigcap_{n=1}^\infty \mathcal{R}_n$.
There is no loss of generality if we also suppose $\|T\|\leq 1$.
Since $\mathcal{A}_n$ is weakly dense in $\mathcal{R}_n$, for every $n$ there exists $A_n\in\mathcal{A}_n$ with
$\|A_n\|\leq 1$ and
\begin{equation}\label{inequality}
|\langle (T-A_n)\eta_i,\eta_j \rangle_q|<\frac{1}{n}\, ,\quad \textrm{for}\,\, i,j=1, 2, \ldots n
\end{equation}
We next rewrite each $A_n$ as $A_n=\lambda_n I+Q_n$, where $\lambda_n\in\bc$ and $Q_n$
is a finite linear combination of Wick-ordered words in $l_h, l^\dag_k$ with
$|h|, |k|\geq n$.\\
We claim that for any fixed $i, j\in\bn$, the inner product $\langle Q_n\eta_i, \eta_j \rangle_q$  eventually vanishes 
(w.r.t. $n$).
Thanks to the equality $\langle A_n\zeta_q, \zeta_q\rangle_q=\lambda_n+\langle Q_n\zeta_q, \zeta_q\rangle_q=\lambda_n$ for
$n$ large enough, we see that the sequence $\{\lambda_n\}_{n\in\bn}$ is bounded with
$|\lambda_n|\leq 1$ for $n$ big enough. But then the sequence of operators $Q_n$ is uniformly bounded as well, and in particular it converges to $0$ in the weak operator topology. 
Up to extracting a subsequence, we can assume that $\{\lambda_n\}_{n\in\bn}$ converges to some
$\lambda\in\bc$. Taking  the limit of \eqref{inequality} for $n\rightarrow\infty$, we finally obtain
$$\langle T\eta_i, \eta_j\rangle_q=\lambda\langle \eta_i,  \eta_j \rangle_q\, ,\quad\textrm{for any}\, i, j\in\bn$$
hence $T=\lambda I$ as the set $\{\eta_i: i\in\bn\}$ is total in the Hilbert space $\G_q(\ell^2(\bz)))$.\\
All is left to do is prove the claim. Fix $i, j\in\bn$, and let us first assume that
$\eta_i= e_{r_1}\otimes\cdots\otimes e_{r_k}$ and $\eta_j=e_{h_1}\otimes\cdots\otimes e_{h_m}$. If we set $N_0:=\max\{|r_1|, ..., |r_k|, |h_1|, ..., |h_m| \}$, for
$n>N_0$ we have $\langle Q_n\eta_i,\eta_j \rangle_q=0$. Indeed, it is enough to observe that
for any Wick-ordered word $W=l^\natural_{p_1}\cdots l^\natural_{p_s}$
with $|p_l|>N_0$ for $l=1, 2, \ldots, s$, where $\natural\in\{1, \dag\}$, the equality
$\langle W(e_{r_1}\otimes\cdots\otimes e_{r_k}),e_{h_1}\otimes\cdots\otimes e_{h_m} \rangle_q=0$
is satisfied. Now, if at least one annihilator shows up in $W$, then $l^\natural_{p_s}=l_{p_s}$ because
$W$ is Wick ordered, which means $l_{p_s}(e_{r_1}\otimes\cdots\otimes e_{r_k})=0$ by \eqref{qann}. Thus
one has
$$\langle W(e_{r_1}\otimes\cdots\otimes e_{r_k}), e_{h_1}\otimes\cdots\otimes e_{h_m} \rangle_q=0. $$
If no annihilator shows up, then $W=l^\dag_{p_1}\cdots l^\dag_{p_s}$, and accordingly 
\begin{align*}
&\langle W(e_{r_1}\otimes\cdots\otimes e_{r_k}), e_{h_1}\otimes\cdots\otimes e_{h_m} \rangle_q\\
&=\langle l^\dag_{p_1}\cdots l^\dag_{p_s}(e_{r_1}\otimes\cdots\otimes e_{r_k}),e_{h_1}\otimes\cdots\otimes e_{h_m} \rangle_q\\
&=\langle l^\dag_{p_2}\cdots l^\dag_{p_s}(e_{r_1}\otimes\cdots\otimes e_{r_k}),l_{p_1}(e_{h_1}\otimes\cdots\otimes e_{h_m})\rangle_q=0.
\end{align*}
Finally, if one of $\eta_i$ and $\eta_j$ is the vacuum vector $\zeta_q$ the above argument continues to work due to
\eqref{qann} once again. 
\end{proof}

\begin{rem}\label{remq}
It is worth pointing out that an analogue of the Hewitt-Savage theorem
does not hold for our $q$-deformed processes, for 
the tail algebra $\xi_{\om_q}^\perp$ is strictly contained in $\pi_{\om_q}(\gar_q)''_{\bp_\bz}$.
To see this, we start by recalling that $\pi_{\om_q}(\gar_q)''=\mathcal{B}(\G_q(\ell^2(\bz)))$, see
\cite[Proposition 3.6 ]{DyF}. But then $\pi_{\om_q}(\gar_q)''_{\bp_\bz}$ is given
by all bounded operators on  $\G_q(\ell^2(\bz))$ that commute
with $U_\sigma$, $\sigma\in\bp_\bz$, where $U_\sigma$ is the unitary operator given by
$U_\sigma\zeta_q=\zeta_q$ and $U_\sigma(e_{i_1}\otimes\cdots \otimes e_{i_k})=e_{\sigma(i_1)}\otimes\cdots\otimes e_{\sigma(i_k)}$, for any $k\in\bn$ and $i_1, \ldots, i_k\in\bz$.
In other words, $\pi_{\om_q}(\gar_q)''_{\bp_\bz}$ is equal to $\{U_\sigma: \sigma\in\bp_\bz\}'$, which is easily seen to contain all operators of the form $\sum_{n=0}^\infty\lambda_n P_n$, where
$P_n$ is the orthogonal projection onto $\ch_n$, the $n$-particle space, with the convention
$\ch_0:= \bc\zeta_q$.\\
We would like to point out that the stationary algebra $\pi_{\om_q}(\gar_q)''_\bz$ coincides
with $\{U_\tau\}'$, where $U_\tau$ is the unitary operator acting on $(\G_q(\ell^2(\bz))$
as $U_\tau\zeta_q=\zeta_q$ and $U_\tau(e_{i_1}\otimes\cdots\otimes e_{i_k})=e_{i_1+1}\otimes\cdots\otimes e_{i_k+1}$, for any $k\in\bn$ and $i_1, \ldots, i_k\in\bz$. In particular, $U_\tau$ belongs to the stationary algebra.
This shows that the stationary algebra and the exchangeable algebra are not the same, for $U_\tau$ does not
commute with all $U_\sigma$'s, as is easily realized by taking $\sigma$ as the transposition $(1, 2)$.
In other words, the theorem by Olshen in \cite{Ols} ceases to hold for $q$-deformed processes.

\end{rem}
However, at the $C^*$-algebra level the exchangeable algebra is trivial.
More precisely, we have the following.

\begin{prop}\label{exq}
The exchangeable algebra 
$$\gar_q^{\bp_\bz}:=\{X\in\gar_q: U_\sigma X U_\sigma^*=X,\,\textrm{for all}\,\, \sigma\in\bp_\bz\}$$ reduces to $\bc I$.
\end{prop}

\begin{proof}
We shall argue by contradiction.
Suppose that $\gar_q^{\bp_\bz}$ is not trivial. Then there exist at least 
two different states $\om_1$ and $\om_2$ on it. Since
the group $\bp_\bz$ is amenable (being a direct limit of finite groups), there exist 
$\bp_\bz$-invariant states $\widetilde{\om_1}$ and $\widetilde{\om_2}$ on $\gar_q$ which restricts
to $\gar_q^{\bp_\bz}$ as $\om_1$ and $\om_2$, respectively. This is absurd because
 $\widetilde{\om_1}\neq\widetilde{\om_2}$, but $\om_q$ is the only $\bp_\bz$-invariant state on $\gar_q$.
\end{proof}
Since $\bz$ is amenable, the same argument as in the proof of Proposition \ref{exq} shows that 
$\gar_q^\bz:=\{X\in\mathfrak{R}_q: U_\tau X U_\tau^*=X\}$ is trivial as well.
Finally, $\gar_q^{\bl_\bz}:=\{X\in\mathfrak{R}_q: \Lambda_g(X)=X\,\, \textrm{for all}\,\, g\in \bl_\bz\}$ is also trivial because $\gar_q^{\bl_\bz}\subseteq \gar_q^\bz$.

\appendix
\section{}

We provide here a result that allows for a full reconstruction of classical (real-valued)
stochastic processes from the general setting of $C^*$-algebras and their representation theory.
To this end, in the definition of a quantum stochastic process, the sample algebra
$\ga$ must be chosen to be $C_0(\br)$, the $C^*$-algebra of all continuous functions on the real line 
vanishing at infinity. Furthermore, the algebra $\bigvee_{j\in\bz}\iota_j (C_0(\br))$
needs to be assumed commutative as well. The result is actually more or less known. Nevertheless, we
state it in a possibly novel form, which is more suited to the language of Quantum Probability.
We keep the notation we established in Section \ref{prel}.

\begin{prop}
Let $(C_0(\br), \ch, \{\iota_j\}_{j\in\bz}, \xi)$ be
a stochastic process such that the von Neumann algebra $\bigvee_{j\in\bz}\iota_j (C_0(\br))$ is commutative.
Then there exists a commuting family $\{A_j: j\in\bz\}$ of
(possibly unbounded) self-adjoint operators on $\ch$ such that for every $j\in\bz$ one has
$\iota_j(f)=f(A_j)$, $f\in C_0(\br)$.\\
In addition there exist a probability measure $\mu$ on $(\br^\bz, \mathfrak{C})$ and a unitary $U: \ch\rightarrow L^2(\br^\bz, \mu)$ such that:
\begin{enumerate}
\item $U\xi=[1]_\mu$
\item$UA_jU^*=M_j$, for every $j\in\bz$, where $M_j$ is the operator acting on $L^2(\br^\bz, \mu)$
as the multiplication by $X_j$.
\end{enumerate}
\end{prop}

\begin{proof}
Let $\cb_b(\br)$  be the $C^*$-algebra of all bounded Borel functions on $\br$.
As is known, each $\iota_j$ can be extended to a $*$-representation of $\cb_b(\br)$, which we
denote by $\widetilde{\iota_j}$. 
By a standard approximation argument, one easily sees that the ranges of the
extended representations continue to commute with each other.\\
For any fixed $j$, set $E_j(\Delta):=\widetilde{\iota_j}(\chi_\Delta) $, where
$\Delta$ is a Borel subset of the real line. For every $j\in\bz$, the family of
orthogonal projections $\{E_j(\Delta): \Delta\subset\br\, \textrm{ is a Borel subset}\}$
is a resolution of the identity, and as such it defines a self-adjoint operator 
$A_j:=\int_{\br}\lambda{\rm d}E_j(\lambda)$.
Note that the operators $A_j$, $j\in\bz$, commute with one another because their spectral
projections do.\\
Fix $n\in\bn$ and $j_1, j_2, \ldots, j_n\in\bz$. Consider the  bounded linear functional
$\varphi: C_0(\br^n)\rightarrow\bc$ given by
$\varphi(f)=\langle f(A_{j_1}, A_{j_2}, \ldots, A_{j_n})\xi, \xi \rangle$
for any $f\in C_0(\br^n)$. By the Riesz-Markov theorem there exists a Borel probability measure 
$\mu_{j_1, j_2, \ldots, j_n}$ on $\br^n$ such that 
\begin{eqnarray}\label{findim}
\langle f(A_{j_1}, A_{j_2}, \ldots, A_{j_n})\xi, \xi \rangle= \int_{\br^n} f{\rm d}\mu_{j_1, j_2, \ldots, j_n}, \, f\in C_0(\br^n).
\end{eqnarray}
Observe that the family $\{\mu_{j_1, j_2, \ldots, j_n}: n\in\bn, j_1, j_2, \ldots, j_n\in\bz\}$
satisfies the consistency conditions of Kolmogorov's theorem. Therefore, there exists
a probability measure $\mu$ on $(\br^\bz, \mathfrak{C})$ having the family above as its finite-dimensional
distributions.\\
Our next aim is to define a unitary $U$ from $\ch$ to $L^2(\br^\bz, \mu)$. To this end, first note that
the linear subspace 
$$\mathcal{D}:=\{f(A_{j_1}, A_{j_2}, \ldots, A_{j_n})\xi: n\in\bn, f\in C_0(\br^n), j_1, j_2, \ldots, j_n\in\bz\}$$
is dense in $\ch$ because $\xi$ is a cyclic vector for $\bigvee_{j\in\bz}\iota_j(C_0(\br))$.
We define $U_0: \mathcal{D}\rightarrow L^2(\br^\bz, \mu)$ by
\begin{eqnarray}\label{U0}
U_0f(A_{j_1}, A_{j_2}, \ldots, A_{j_n})\xi:= [f]_\mu\,,
\end{eqnarray}
where $[f]_\mu$ is the $\mu$-equivalence class of the function 
$f(X_{j_1}, X_{j_2}, \ldots, X_{j_n})$, and $X_j: \br^\bz\rightarrow\br$
is for every $j$ the $j$-th coordinate function, that is $X_j(x)=x_j$, $x=(x_j)\in\br^\bz$.
Thanks to \eqref{findim} one sees at once that $U_0$ is an isometry. By density of
$\mathcal{D}$ in $\ch$, $U_0$ can be extended to an isometry $U$ defined on the whole $\ch$.
By standard approximation arguments the range of $U$ is seen to be dense
in  $L^2(\br^\bz, \mu)$, hence $U$ is a unitary. 
In order to prove the equality $U\xi = [1]_\mu$,
we consider a bijection $g:\bn\rightarrow\bz$ and a sequence of
cylinders sets $C_n=\{x\in\br^\bz: x_{g(1)}\in B_1, \ldots, x_{g(n)}\in B_n \}$ such that
$\mu(C_n)\geq 1-\frac{1}{n}$ with $B_i\subset\br$ being suitable bounded Borel subsets.
Let now be $\{h_n\}_{n\in\bn}$ be a sequence of functions such that
 $h_n\in C_c(\br^n)$ with
$0\leq h_n\leq 1$ and $h_n(x_1, \ldots, x_n)=1$ for all
$(x_1, \ldots, x_n)\in B_1\times\cdots\times B_n$, for every $n$.
By evaluating \eqref{U0} on $h_n(A_{g(1)}, \ldots, A_{g(n)})\xi$, we get 
$$
U_0h_n(A_{g(1)}, \ldots, A_{g(n)})\xi= [h_n]_\mu\,.
$$
By construction the sequence $\{k_n\}_{n\in\bn}\subset L^2(\br^\bz, \mu)$ given by
$k_n:= h_n(X_{g(1)}, X_{g(2)}, \ldots, X_{g(n)})$ converges to
$[1]_\mu$ in the $\|\cdot\|_{L^2}$-norm. Therefore, one can take the limit as $n\rightarrow\infty$
of the above equality and gets $U\xi=[1]_\mu$.\\
All is left to do is prove the equality
$UA_jU^*=M_j$, for every $j\in\bz$. Fix $j\in\bz$ and define 
$$\mathcal{D}_j:=\{f(A_j, A_{j_1}, \ldots, A_{j_n})\xi: n\geq 0, j_1, \ldots. j_n\in\bz, f\in C_c(\br^{n+1})\}.$$
We aim to show that $\mathcal{D}_j$ is a core per $A_j$. We first note that
$\mathcal{D}_j$ is contained in the domain of $A_j$. Indeed, if $x=f(A_j, A_{j_1}, \ldots, A_{j_n})\xi$,
where $f$ is in $C_c(\br^{n+1})$ one has
$\int_\br \lambda_j^2 {\rm d}(E(\lambda_j)x, x)<\infty$ because the measure ${\rm d}(E(\lambda_j)x, x)$ is
by construction compactly supported. Second, the linear subspace
$\mathcal{D}_j$ is easily seen to be dense in $\ch$. Finally, $\mathcal{D}_j$
is invariant for the one-parameter group $U_j(t):=e^{itA_j}$ since for any $f\in C_0(\br^{n+1})$
the function $\br^{n+1}\ni(\lambda_j, \lambda_{j_1}, \ldots, \lambda_{j_n})\rightarrow e^{it\lambda_j}f(\lambda_j, \lambda_{j_1}, \ldots, \lambda_{j_n})\in\bc$ still has compact support. From Theorem VIII.11 in \cite{RS} it follows
that $\mathcal{D}_j$ is a core for $A_j$.\\
The equality $UA_j=M_jU$ is trivially satisfied on $\mathcal{D}_j$.
Let now $\varphi$ be in $\mathcal{D}(A_j)$. Then there exists
a sequence $\{\varphi_n\}_{n\in\bn}$ in $\mathcal{D}_j$ such that
$\varphi_n\rightarrow\varphi$ and $A_j\varphi_n\rightarrow A_j\varphi$. From the equality
$UA_j\varphi_n=M_jU\varphi_n$, $n\in\bn$, we see that the sequence
$M_jU\varphi_n$ converges to $UA_j\varphi$. Because $M_j$ is closed, we must have that
$U\varphi$ is in $\mathcal{D}(M_j)$ and  $UA_j\varphi=M_jU\varphi$.
In other words, we have $A_j\subset U^*M_jU$, and so
$A_j=U^*M_jU$ as $A_j$ and $U^*X_jU$
are both self-adjoint.
\end{proof}

\begin{rem}
Under the general hypotheses we are working with, a common core for all
operators $A_j$ may fail to exist, for  the intersection $\bigcap_{j\in\bz} \mathcal{D}_j$ is not even necessarily dense
in $\ch$.
\end{rem}

\section*{Acknowledgments}\noindent

We acknowledge the support of Italian INDAM-GNAMPA.

\end{document}